\documentclass[11pt]{amsart}
\usepackage{graphicx, color, epsfig}
\usepackage{amscd}
\usepackage{amsmath,amsthm,empheq}
\usepackage{amsfonts}
\usepackage{amssymb}
\usepackage{mathrsfs}
\usepackage{graphicx}
\usepackage[all]{xy}

\numberwithin{equation}{section}

\usepackage[latin1]{inputenc}

\usepackage[english]{babel}

\newtheorem{corollary}{Corollary}[section]
\newtheorem{definition}[corollary]{Definition}

\newtheorem{lemma}[corollary]{Lemma}

\newtheorem{thm}[corollary]{Theorem}

\newfont{\sBlackboard}{msbm10 scaled 900}

\newcommand{\mylabel}[1]{\label{#1}
            \ifx\undefined\stillediting
            \else \fbox{$#1$}\fi }
\newcommand{\BE}{\begin{equation}}

\newcommand{\EEQ}{\end{equation}}
\newcommand{\rfb}[1]{\mbox{\rm
   (\ref{#1})}\ifx\undefined\stillediting\else:\fbox{$#1$}\fi}

\newfont{\Blackboard}{msbm10 scaled 1200}

\newfont{\roma}{cmr10 scaled 1200}

\newcommand{\bb}{\begin{equation}}
\newcommand{\bbb}{\end{equation}}

\newcommand{\mm}    {{\hbox{\hskip 0.5pt}}}

\newcommand{\bluff} {{\hbox{\raise 15pt \hbox{\mm}}}}
scaled\magstep2

\makeatletter
\def\section{\@startsection {section}{1}{\z@}{-3.5ex plus -1ex minus
    -.2ex}{2.3ex plus .2ex}{\large\bf}}
\makeatother

\begin{document}
\title[super-critical nonlinearite]{ A note on the existence results for Schr\"odinger-Maxwell system with super-critical nonlinearitie}
\author[A. Bahrouni]{Anouar Bahrouni}
\address[A. Bahrouni]{Mathematics Department, University of Monastir,
Faculty of Sciences, 5019 Monastir, Tunisia} \email{\tt
bahrounianouar@yahoo.fr}
\begin{abstract}
The paper considers the following Schr\"odinger-Maxwell system with
supercritical nonlinearitie,
\begin{equation}\label{man}
\begin{cases}
-\Delta u+K(x) \phi u =|u|^{p-1}u+h(x), \ \ \ \ \ \ \ \ \mbox{in} \ \ \Omega,\\
-\Delta \phi = K(x)u^{2}, \ \ \ \ \ \ \ \ \ \ \ \ \ \ \ \ \ \ \ \ \
\ \ \ \ \ \ \ \ \  \  \mbox{in} \ \  \Omega,\\
\phi= u=0, \ \ \ \ \ \ \ \ \ \ \ \ \ \ \ \ \ \ \ \ \ \ \ \ \ \ \ \ \
\  \ \ \ \ \ \ \ \    \mbox{in} \ \ \partial\Omega,
\end{cases}
\end{equation}
where $\Omega \subset \mathbb{R}^{3}$  is a bounded domain with
smooth boundary, $1<p\ \ \mbox{and} \ \
K,h\in{L^{\infty}}\left(\Omega\right)$.
 We prove the existence of at least one non-trivial weak solution. This
result is already known for the subcritical case. In this paper, we
extend it to the supercritical values of $p$ as well. We use a new
variational principle to prove our result.
\end{abstract}

\maketitle
\section{Introduction and main results}
In the present paper we study the existence of solution for the
following electrostatic nonlinear Schr\"odinger-Maxwell equations
also known as nonlinear Schr\"odinger-Poisson system
\begin{equation}\label{main}
\begin{cases}
-\Delta u+K(x) \phi u =|u|^{p-1}u+h(x), \ \ \ \ \ \ \ \ \mbox{in} \ \ \Omega,\\
-\Delta \phi = K(x)u^{2}, \ \ \ \ \ \ \ \ \ \ \ \ \ \ \ \ \ \ \ \ \
\ \ \ \ \ \ \ \ \  \  \mbox{in} \ \  \Omega,\\
\phi= u=0, \ \ \ \ \ \ \ \ \ \ \ \ \ \ \ \ \ \ \ \ \ \ \ \ \ \ \ \ \
\  \ \ \ \ \ \ \ \    \mbox{in} \ \ \partial\Omega,
\end{cases}
\end{equation}
where $\Omega \subset \mathbb{R}^{N}$, ($N=3$),  is a bounded domain
with
smooth boundary, $1<p$ and $K,h\in L^{\infty}(\Omega)$.\\
Similar system arises in many mathematical physics contexts while
looking for existence of standing waves for the nonlinear
Schr\"odinger
 equations interacting with an unknown electrostatic field. For more details on the physics aspect we refer the reader to \cite{1,2}. \\
 In recent years, a number of papers have contributed to investigate
the existence of solutions of \eqref{main}. We can cite
\cite{11,3,al,ab,1111,15,2,14,12,7} and the references therein. For
the case where $\Omega$ is a bounded domain, we would like to cite
the papers of Ruiz and Siciliano \cite{r} and Siciliano\cite{gs}. In
all those papers, the solutions found are in the case where $1<p<5$.
In the unbounded case, Ambrosetti and Ruiz \cite{3} and Ruiz
\cite{7} considered problem
\begin{equation}\label{mmain}
\begin{cases}
-\Delta u+V(x)u+\mu \phi u =|u|^{p-1}u, \ \ \ \ \ \ \ \ \ \ \ \  \  \mbox{in} \ \ \mathbb{R}^{3},\\
-\Delta \phi = 4 \pi^{2}u^{2}, \ \ \ \ \ \ \ \ \ \ \ \ \ \ \ \ \ \ \
\ \ \ \ \ \ \ \ \ \ \     \mbox{in} \ \ \mathbb{R}^{3}.
\end{cases}
\end{equation}
By working in the radial functions subspace of
$H^{1}(\mathbb{R}^{3})$ and taking $1<p<5$ and $V(x)=1$, they were
able to obtain
the existence and multiplicity results.\\
In \cite{12}, Jiang and Zhou have treated the problem \eqref{mmain}
where $\Omega=\mathbb{R}^{3}$, $K=\lambda>0, \ \ 1<p<6$ and $V$
change sign.
 With further assumptions on $V$, the authors have proved that problem \eqref{main} has at least a positive solution.\\
If $0<p<1$, Bahrouni and Ounaies \cite{ab} has treated system
\eqref{main} where $\Omega=\mathbb{R}^{3}$. By using the variational
method, they have proved that problem
\eqref{main} has infinitely many solutions. We also refer to \cite{baraou,4,r,gs,6}.\\
Motivated by papers above, we are interested in finding solution for
system \eqref{main}, by assuming only that $p>1$. Our methodology is
based on a new variational principle
established in \cite{k,m}.\\
In order to state our main results, we give the following assumptions:\\
$\left(K\right) \ \ K\in L^{\infty}\left(\Omega\right)$ and
$K(x)\geq 0, \ \ \forall x \in \Omega$.\\
$\left(H\right) \ \ h\in L^{\infty}\left(\Omega\right)$ and $h(x)> 0, \ \ \forall x \in \Omega.$\\
Now we can state our result.
\begin{thm}\label{mainthm}
Assume that $(H)$ and $(K)$ hold. Suppose that $p>1$. Then, there
exists $m>0$ such that if $\|h\|_{L^{N}(\Omega)}\leq m$, problem
\eqref{main} admits at least one nontrivial solution.
\end{thm}
 The remainder of our paper is organized as follows.
In section $2$, some preliminary results are presented. While
section $3$ is dedicated to the proof of Theorems \ref{mainthm}.
\section{Variational settings and preliminary results}
First, we give some notations. For $1\leq m< +\infty, \ \
L^{m}\left(\Omega\right)$ is the usual Lebesgue space with the norm
$$\left\|u\right\|_{L^{m}(\Omega)}=\displaystyle \left(\int_{\Omega} \left| u \right|^{m}dx\right)^{\frac{1}{m}}.$$
Hereafter, the space $E=H^{1}_{0}\left(\Omega\right)\cap
L^{p+1}\left(\Omega\right)$ is endowed with the following norm
$$\left\|u\right\|=\left\|\nabla u\right\|_{2}+\left\|u\right\|_{p+1}.$$
We shall now recall some results for the Sobolev space required in
the sequel (see \cite{gi,k}).
\begin{lemma}\label{sob}
Let $\Omega$ be a bounded $C^{0,1}$ domain in $\mathbb{R}^{N}$. Then: \\
$i)$ If $0\leq m< k-\frac{N}{p}<m+1$, the space $W^{k,p}(\Omega)$ is
continuously imbedded in $C^{m,\alpha}(\overline{\Omega})$,
$\alpha=k-\frac{N}{p}-m$, and compactly imbedded in
$C^{m,\beta}(\overline{\Omega})$ for any $\beta <\alpha.$\\
$ii)$ $u\rightarrow \|\Delta u\|_{L^{N}(\Omega)}$ is an equivalent
norm on $H^{1}_{0}(\Omega)\frown W^{2,N}(\Omega)=E\frown
W^{2,N}(\Omega)$.
\end{lemma}
 An important fact involving system \eqref{main} is that
this class of system can be transformed into a Schr\"odinger
equation (see, for instance \cite{2,7}), with a nonlocal term. By
the Lax-Milgram Theorem, given $u\in E$, there exists a unique
$\phi_{u}\in H^{1}_{0}(\Omega) $ such tha $-\Delta
\phi_{u}=K(x)u^{2}.$ By using standard arguments, we have that
$\phi_{u}$ verifies the following properties ( see \cite{2,7});
\begin{lemma}\label{phi}
For any $u\in E$, we have\\
$1)$ there exists $C>0$ such that $\|\phi_{u}\|\leq
C\|u\|^{2}$.\\
$2)$ $\phi_{u}\geq 0, \ \ \phi_{tu}=t^{2}\phi_{u}, \forall \ \
t\geq 0 \ \ \mbox{and} \ \ u\in E$.\\
$3)$ If $u_{n}\rightharpoonup u$ in $E$, then
$\phi_{u_{n}}\rightharpoonup \phi_{u}$ in $E$ and $\displaystyle
\lim_{n\rightarrow
+\infty}\int_{\Omega}\phi_{u_{n}}{u_{n}}^{2}dx=\int_{\Omega}\phi_{u}u^{2}dx$.\\
$4)$ If $u\in W^{2,N}(\Omega) \frown H^{1}_{0}(\Omega)$, then
$\phi_{u}\in W^{2,N}(\Omega) \frown H^{1}_{0}(\Omega)$.
\end{lemma}
So, the functional $I: E\rightarrow \mathbb{R}$,
$$I(u)=\frac{1}{2}\displaystyle \int_{\Omega}\left|\nabla u\right|^{2}dx+\frac{1}{4}\displaystyle
 \int_{\Omega}K(x)\phi_{u}u^{2}dx-\frac{1}{p+1}\displaystyle \int_{\Omega}\left|u\right|^{p+1}dx-\displaystyle \int_{\Omega} h(x)u dx, \ \ \forall u\in E$$
is $C^{1}$ on $E$ and $$\left\langle
I^{'}(u),\varphi\right\rangle=\displaystyle \int_{\Omega}\nabla
u\nabla \varphi dx+ \displaystyle \int_{\Omega}K(x)\phi_{u}u \varphi
dx-\displaystyle \int_{\Omega}\left|u\right|^{p-1}u \varphi
dx-\displaystyle \int_{\Omega} h(x)\varphi dx,$$
  for all $u\in E$ and $\varphi \in H^{1}_{0}\left(\Omega\right)$.
It is also known that $\left(u,\phi\right)\in E \times
H^{1}_{0}\left(\Omega\right)$ is a solution of \eqref{main} if and
only if $u\in E$ is a critical point of the functional $I$, and
$\phi=\phi_{u}$, see for instance \cite{1}. let us recall that a
Palais-Smale sequence for the functional I, for short we write (PS)-
sequence, is a sequence $\left(u_{n}\right)$ such that
$$\left(I(u_{n})\right) \ \ \mbox{is bounded in E and} \ \
\left\|I^{'}(u_{n})\right\|_{E^{'}}\rightarrow 0.$$\\ I is said to
satisfy the Palais-Smale condition if any (PS)-sequence possesses a
convergent subsequence in $E$. Now, we recall some important
definitions and results from
\cite{gi}.\\
Let $E$ be a real Banach space.  Let $\psi: E\rightarrow
\mathbb{R}\smile \{\infty\}$ be a proper (i.e. $Dom(\psi))=\{u\in E;
\psi(u)<\infty\}\neq \emptyset $) convex function. The
subdifferential $\partial \psi$ of $\psi$ is defined to be the
following set-valued operator: if $u\in Dom(\psi)$, set
$$\partial \psi (u)=\{u^{'}\in E^{'}; \langle u^{'},v-u\rangle+\psi(u)\leq \psi(v), \ \ \forall v\in E\},$$
and if $u \notin Dom(\psi)$, set $\partial \psi (u)=\emptyset$. If
$\psi$ is G\^ateaux differentiable at $u$, denote by $D\phi(u)$ the
G\^ateaux derivative of $\psi$ at $u$. In this case $\partial \psi
(u)=\{D\psi(u)\}$.\\
The restriction of $\psi$ to $K\subset E$ is denoted by $\psi_{K}$
and defined by
$$\psi_{K}(u)=\psi(u) \ \ \mbox{if} \ \ u\in K \ \ \mbox{and} \ \ \psi_{K}(u)=+\infty \ \ \mbox{if} \ \ u\notin K.$$
Let $J$ be a function on $E$ satisfying the following hypothesis:\\
$(R)$: $J=\psi-\phi$, where $\phi\in C^{1}(E,\mathbb{R})$ and $\psi:
E \rightarrow (-\infty, +\infty]$ is proper, convex and lower semi
continuous.
\begin{definition}\label{def}
A point $u\in E$ is said to be a critical point of $I=\psi-\phi$ if
$u\in Dom(\psi)$ and if it satisfies the inequality $\langle D \phi
(u),u-v\rangle+\psi(v)-\psi(u)\geq 0, \ \ \forall v\in E,$ where
$D\phi(u)$ stands for the derivative of $\phi$ at $u.$
\end{definition}
\begin{lemma}\label{R}
If $I$ satisfies $(R)$, then each local minimum of $I$ is
necessarily a critical point of $I$.
\end{lemma}
\begin{proof}
See \cite{k}.
\end{proof}
 Now, we define the functionals $\phi,\psi:E \rightarrow \mathbb{R}$
by $$\phi(u)=-\frac{1}{4}\displaystyle
 \int_{\Omega}K(x)\phi_{u}u^{2}dx+\frac{1}{p+1}\displaystyle \int_{\Omega}\left|u\right|^{p+1}dx\\+\displaystyle \int_{\Omega}h(x)udx,$$ and  $$\psi(u)=\frac{1}{2}\displaystyle \int_{\Omega}\left|\nabla
 u\right|^{2}dx \ \ \mbox{and} \ \  I_{K}(u)=\psi_{K}(u)-\phi(u).$$
\section{Proof of Theorem \ref{mainthm}}
We now give the following variational principle version applicable
to problem \eqref{main}.
\begin{thm}\label{thmcri}
Let $K\subset E$ be a convex and weakly closed subset of $E$. If the
following two assertions hold:\\
$(i)$ The functional $I_{K}$ has a critical point $u_{1}\in E$ as in
Definition \ref{def}, and;\\
$(ii)$ there exists $u_{2}\in K$ such that
$$\displaystyle \int_{\Omega}\nabla u_{2}\nabla \varphi dx=
-\displaystyle \int_{\Omega}K(x)\phi_{u_{1}}u_{1} \varphi
dx+\displaystyle \int_{\Omega}\left|u_{1}\right|^{p-1}u_{1} \varphi
dx+\displaystyle \int_{\Omega}h(x)\varphi dx, \ \ \forall  \varphi
\in E.$$ Then $u_{1}\in K$ is a weak solution of system
\eqref{main}.
\end{thm}
\begin{proof}
Since $u_{1}$ is a critical point of $I_{K}$, then
\begin{equation}\label{1}
\frac{1}{2}\displaystyle \int_{\Omega}\left|\nabla
v\right|^{2}dx-\frac{1}{2}\displaystyle \int_{\Omega}\left|\nabla
u_{1}\right|^{2}dx\geq \displaystyle
  \int_{\Omega}(-K(x)\phi_{u_{1}}u_{1}+\left|u_{1}\right|^{p-1}u_{1}+h(x))(v-u_{1})dx, \ \ \forall v\in K.
\end{equation}
Invoking assumption $(ii)$ in the theorem, we deduce that
\begin{equation}\label{2}
\displaystyle \int_{\Omega}\nabla u_{2}\nabla
(u_{1}-u_{2})dx=\displaystyle
\int_{\Omega}(-K(x)\phi_{u_{1}}u_{1}+\left|u_{1}\right|^{p-1}u_{1}+h(x))(u_{1}-u_{2})dx.
\end{equation}
Now by substituting $v=u_{2}$ in \eqref{1} and taking into account
\eqref{2}, we obtain
\begin{align}\label{3}
\frac{1}{2}\displaystyle \int_{\Omega}\left|\nabla
u_{2}\right|^{2}dx-\frac{1}{2}\displaystyle
\int_{\Omega}\left|\nabla u_{1}\right|^{2}dx &\geq
\frac{1}{4}\displaystyle
  \int_{\Omega}(-K(x)\phi_{u_{1}}u_{1}+\left|u_{1}\right|^{p-1}u_{1}+h(x))(u_{2}-u_{1})dx \nonumber\\
  &=\displaystyle \int_{\Omega}\nabla u_{2}\nabla
(u_{2}-u_{1})dx.
\end{align}
On the other hand, in view of the convexity of $\psi$, we infer that
\begin{equation}\label{4}
\frac{1}{2}\displaystyle \int_{\Omega}\left|\nabla
u_{1}\right|^{2}dx-\frac{1}{2}\displaystyle
\int_{\Omega}\left|\nabla u_{2}\right|^{2}dx\geq \displaystyle
\int_{\Omega}\nabla u_{2}\nabla (u_{1}-u_{2})dx.
\end{equation}
Using the above pieces of informations, we obtain that
$$\frac{1}{2}\displaystyle \int_{\Omega}\left|\nabla
u_{2}\right|^{2}dx-\frac{1}{2}\displaystyle
\int_{\Omega}\left|\nabla u_{1}\right|^{2}dx= \displaystyle
\int_{\Omega}\nabla u_{2}\nabla (u_{2}-u_{1})dx.$$ This shows that
$$ \frac{1}{2}\displaystyle \int_{\Omega}\left|\nabla u_{2}-\nabla
u_{1}\right|^{2}dx=0.$$ Thus, $$u_{2}=u_{1},$$ for a.e. $x\in
\Omega.$\\
This ends the proof.
\end{proof}
We shall use the above theorem to prove our main result in Theorem
\ref{mainthm}. The convex subset $K\subset E$ required in
Theorem\ref{thmcri} is defined as follows
$$K(r)=\{u\in E: \|u\|_{W^{2,N}(\Omega)}\leq r\},$$
for some $r>0$ to be determined later.
\begin{lemma}
Let $r>0$ be fixed. The set
$$\{u\in E: \|u\|_{W^{2,N}(\Omega)}\leq r\},$$ is a weakly closed in
$E$.
\end{lemma}
\begin{proof}
See \cite{k}.
\end{proof}
In the sequel, we need the following technical lemmas.
\begin{lemma}\label{K}
Let $r>0$ be fixed. Then, there exists $C_{1},C_{2}>0$  such that
$$\|-K(x)\phi_{u}u+|u|^{p-1}u+h(x)\|_{L^{N}(\Omega)}\leq C_{1}r^{3}+C_{2}r^{p}+\|h\|_{L^{N}(\Omega)}, \ \ \forall u\in K(r).$$
\end{lemma}
\begin{proof}
Let $u\in K(r)$. Then, using Lemmas \ref{sob} and \ref{phi} and
H\"older's inequality, we get
\begin{align*}
\|-K(x)\phi_{u}u+|u|^{p-1}u\|_{L^{N}(\Omega)}&\leq
\|K\|_{\infty} \|\phi_{u}u\|_{L^{N}(\Omega)}+\||u|^{p-1}u\|_{L^{N}(\Omega)}\\
&\leq
\|K\|_{\infty}\|u\|_{L^{2N}(\Omega)}\|\phi_{u}\|_{L^{2N}(\Omega)}+\|u\|_{L^{Np}(\Omega)}^{p}\\
&\leq c_{1}\|u\|_{W^{2,N}(\Omega)}
\|\phi_{u}\|_{W^{2,N}(\Omega)}+c_{2}\|u\|_{W^{2,N}(\Omega)}^{p }\\
& \leq C_{1}
\|u\|_{W^{2,N}(\Omega)}^{3}+C_{2}\||u|^{p-1}u\|_{L^{N}(\Omega)}\\ &
\leq C_{1} r^{3}+C_{2}r^{p},
\end{align*}
where $c_{1},c_{2},C_{1},C_{2}>0$. This ends the proof.
\end{proof}
\begin{lemma}\label{r}
Assume that $C_{1}$ and $C_{2}$ are given in Lemma \ref{K}. Then,
there is $r_{1}>0$ such that $C_{1}r^{3}+C_{2}r^{p}\leq \frac{r}{2},
\ \ \forall r\in (0,r_{1}].$ Moreover, if $\|h\|_{L^{N}(\Omega)}\leq
\frac{r_{1}}{2}$, we have
$$C_{1}r_{1}^{3}+C_{2}r_{1}^{p}+\|h\|_{L^{N}(\Omega)} \leq r_{1}.$$
\end{lemma}
\begin{proof}
The proof follows by a straightforward computation.
\end{proof}
\begin{lemma}\label{ii}
Suppose that conditions of Theorem \ref{mainthm} are fulfilled. Let
$r_{1}$ be given in Lemma \ref{r}. Moreover, assume that
$\|h\|_{L^{N}(\Omega)}\leq \frac{r_{1}}{2}$. Then for each $u \in
K(r_{1})$ there exists $v\in K(r_{1})$ such that
\begin{equation}\label{eqii}
\displaystyle \int_{\Omega}\nabla v\nabla \varphi dx+ \displaystyle
\int_{\Omega}K(x)\phi_{u}u \varphi dx=\displaystyle
\int_{\Omega}\left|u\right|^{p-1}u \varphi dx+\displaystyle
\int_{\Omega}h(x)\varphi dx,
\end{equation}
 for all $u\in E$ and $\varphi \in H^{1}_{0}\left(\Omega\right).$ In
 particular, $v\in W^{2,N}(\Omega)\frown H^{1}_{0}(\Omega)$, and
 \begin{equation}\label{ei}
 -\Delta v=-K(x)\phi_{u}u+|u|^{p-1}u+h(x), \ \ \mbox{for a.e} \ \  x\in \Omega.
 \end{equation}
\end{lemma}
\begin{proof}
Using a standard argument, there exists $v\in H^{1}_{0}(\Omega)$
which satisfies \eqref{eqii}. Since the right hand side is an
element in $L^{N}(\Omega)$, it follows from the standard regularity
results that $v\in W^{2,N}(\Omega)\frown H^{1}_{0}(\Omega)$ and
\eqref{ei} holds. Therefore, using Lemmas \ref{K} and \ref{r}, we
deduce that

\begin{align*} \|v\|_{W^{2,N}(\Omega)}=\|\Delta
v\|_{L^{N}(\Omega)}&=\|-K(x)\phi_{u}u+|u|^{p-1}u+h(x)\|_{L^{N}(\Omega)}\\
&\leq r_{1},
\end{align*}
the lemma is proven.
\end{proof}
\textbf{Proof of Theorem \ref{mainthm} completed:}\\
Let $r_{1}>0$ be as in Lemma \ref{r} and define $K= K(r_{1}).$ We
suppose that $\|h\|_{L^{N}(\Omega)} \leq \frac{r_{1}}{2}$.
\\ Consider the following minimizing problem $$\beta=\displaystyle
\inf_{u\in E}I_{K}(u).$$ Hence, by definition of $\psi_{K}$, we
deduce that $$\beta=\displaystyle \inf_{u\in K}I_{K}(u).$$ On the
other hand, using Lemma \ref{K}, we infer that $\beta >-\infty.$
Take $0< e\in K$. For $t\in [0,1]$, we have that $te\in K$ and
therefore
$$I_{K}(te)\leq t(t \displaystyle \int_{\Omega}|\nabla e|^{2}dx+t^{3}\displaystyle \int_{\Omega} \phi_{e}e^{2}dx-
t^{p}\displaystyle \int_{\Omega} |e|^{p+1}dx-\displaystyle
\int_{\Omega}  h(x)e dx).$$ Since $h,e>0$, we can conclude that
$\beta<0$. Now suppose that $(u_{n})$ is a sequence in $E$ such that
$I_{K}(u_{n})\rightarrow \beta$. So the sequence is bounded and we
can conclude by the definition of $I_{K}$ that the sequence is
bounded in $W^{2,N}(\Omega)$. Using standard results in Sobolev
spaces, after passing to a subsequence if necessary, there exists
$u_{1}\in K$ such that $u_{n}\rightharpoonup u_{1}$ in
$W^{2,N}(\Omega)$ and strongly in $E$. Therefore,
$$\beta=I_{K}(u_{1})<0.$$ Then, by Lemma \ref{R}, we conclude that
$u_{1}$ is a nontrivial critical point of $I_{K}$. Now, by Lemma
\ref{ii} together with the fact that $u_{1}\in K(r_{2})$ we obtain
that there exists $u_{2}\in K$ such that $$-\Delta
u_{2}=-\phi_{u_{1}}u_{1}+|u_{1}|^{p-1}u_{1}+h.$$\\
Combining the above pieces of informations and applying Theorem
\ref{thmcri}, we conclude that $u_{1}$ is a nontrivial solution of
problem \eqref{main}.

\end{document}